\documentclass[conference]{IEEEtran}
\IEEEoverridecommandlockouts
\usepackage{cite}
\usepackage{amsmath,amssymb,amsfonts}
\usepackage{graphicx}
\usepackage{textcomp}
\usepackage{xcolor}

\usepackage{fullpage}
\usepackage{amsthm}
\usepackage{amsmath,amsfonts,amssymb,multirow}
\usepackage{algorithm} % algorithm
\usepackage{algpseudocode} %algpseudocode
\usepackage{graphicx} % \graphicspath
\usepackage{csquotes} % displayquote
\usepackage{placeins} % \FloatBarrier
\usepackage[normalem]{ulem} % \sout
\usepackage{xcolor}
\usepackage{subcaption}
\usepackage{marvosym} % \Letter
\usepackage[capitalize]{cleveref}

\usepackage{tikz}
\usetikzlibrary{arrows,calc}
\usepackage{url}
\usepackage{color}

\usepackage{paralist} % http://mirrors.ibiblio.org/CTAN/macros/latex/contrib/paralist/paralist.pdf
\usepackage{bbm} 
\usepackage{wrapfig} % \begin{wrapfigure} ... \end{wrapfigure}
\usepackage{comment} % \comment \endcomment
\usepackage{palatino} % palatino font

\usepackage{aligned-overset} % \overset{\text{Explanation}}&{=} within align environment.

\providecommand{\oline}[1]{\mkern 1.5mu\overline{\mkern-1.5mu#1}}

\renewcommand{\hbar}{\oline{h}}

\newcommand{\va}{\mathbf{a}}
\newcommand{\vb}{\mathbf{b}}

\newcommand{\ve}{\mathbf{e}}

\newcommand{\vr}{\mathbf{r}}

\newcommand{\vx}{\mathbf{x}}

\newcommand{\mA}{\mathbf{A}}

\newcommand{\norm}[1]{\left\| #1 \right\|}

\newcommand{\abs}[1]{\left| #1 \right|}

\newtheorem{theorem}{Theorem}[section]

\newtheorem{lemma}[theorem]{Lemma}
\newtheorem{corollary}[theorem]{Corollary}

\newtheorem{remark}[theorem]{Remark}

\usepackage[colorinlistoftodos,bordercolor=orange,backgroundcolor=orange!20,linecolor=orange,textsize=scriptsize]{todonotes}

\newcommand{\xinit}{\vx_0}
\newcommand{\xk}{\vx_k}
\newcommand{\xkpo}{\vx_{k+1}}
\newcommand{\xopt}{\vx^\star}

\newcommand{\einit}{\ve_0}
\newcommand{\ek}{\ve_k}
\newcommand{\ekpo}{\ve_{k+1}}

\def\BibTeX{{\rm B\kern-.05em{\sc i\kern-.025em b}\kern-.08em
    T\kern-.1667em\lower.7ex\hbox{E}\kern-.125emX}}
\begin{document}

\title{QuantileRK: Solving Large-Scale Linear Systems with Corrupted, Noisy Data}

% \dn{(i) Check that Kaczmarz is capitalized in all bib
% (ii) Need less self-citations and more others -- in particular the first ref to [6] should include the original Kaczmarz paper}
\author{\IEEEauthorblockN{Benjamin Jarman\IEEEauthorrefmark{1},
Deanna Needell\IEEEauthorrefmark{2}}
\IEEEauthorblockA{Department of Mathematics, University of California, Los Angeles, Los Angeles, CA \\
\IEEEauthorrefmark{1}bjarman@math.ucla.edu,
\IEEEauthorrefmark{2}deanna@math.ucla.edu}}
\maketitle

\begin{abstract}
Measurement data in linear systems arising from real-world applications often suffers from both large, sparse corruptions, and widespread small-scale noise. This can render many popular solvers ineffective, as the least squares solution is far from the desired solution, and the underlying consistent system becomes harder to identify and solve. QuantileRK is a member of the Kaczmarz family of iterative projective methods that has been shown to converge exponentially for systems with arbitrarily large sparse corruptions. In this paper, we extend the analysis to the case where there are not only corruptions present, but also noise that may affect every data point, and prove that QuantileRK converges with the same rate up to an error threshold. We give both theoretical and experimental results demonstrating QuantileRK's strength.
\end{abstract}

\section{Introduction}
From medical imaging \cite{hounsfield_CAT}, to image reconstruction and signal processing \cite{Herman1993AlgebraicRT, Feichtinger1992NewVO}, to modern data science and statistical analysis \cite{leskovec_rajaraman_ullman_2014}, solving systems of linear equations, has long been a central problem in applied mathematics. Such systems will often be large, overdetermined, and consistent: we consider the system
$
\mA \vx = \vb,
$
where $\mA \in \mathbb{R}^{m \times n}$, $\vb \in \mathbb{R}^m$, and $m \geq n$, with solution $\xopt$. 

A practical challenge is that measurement data often becomes damaged during collection, transmission, or storage, violating consistency. Two important types of damage are 
\begin{itemize}
    \item \emph{corruption}; large errors due to faulty software, hardware, or mismeasurement, affecting a small fraction of data, and
    \item \emph{noise}; small errors due to imprecision or processing that may affect every measurement.
\end{itemize}

The Randomized Kaczmarz (RK) method \cite{ogkacz, vershstrohkacz} is a popular iterative projective method for large, overdetermined, consistent systems due to its exponential convergence and low memory requirements. An initial guess $\xinit$ is iteratively projected onto randomly chosen hyperplanes corresponding to solution spaces to rows of the system. More precisely, letting $\va_1, \cdots, \va_m$ be the rows of $\mA$, the $k$th iterate is computed as
\[
\xk = \vx_{k-1} + \frac{b_i - \va_i \vx_{k-1}}{\norm{\va_i}^2}\va_i^\top,
\]
where row $i$ has been chosen with probability proportional to its Euclidean norm (denoted $\norm{\cdot}$).

Strohmer and Vershynin \cite{vershstrohkacz} showed that RK converges exponentially in expectation. This was extended to the noisy case in \cite{needellnoisyrk} where a vector of noise $\vr$ is added to the measurement data $\vb$. In this case, exponential convergence is still achieved up to an error horizon depending on the size of the noise. Namely, letting $\ek := \xk - \xopt$ be the error at the $k$th iteration,
\[
\mathbb{E}\norm{\ek}^2 \leq \left(1 - \frac{\sigma_{\min}^2(\mA)}{\norm{\mA}_F^2}\right)\norm{\einit}^2 + \frac{\norm{\mA}_F^2}{\sigma_{\min}^2(\mA)}\norm{\vr}^2,
\]
where $\sigma_{\min}(\mA)$ is the smallest singular value of $\mA$, and $\norm{\cdot}_F$ is the Fr\"obenius norm.

Variants of RK, including those involving multi-row projections or greedy row selection, have been shown to exhibit similar robustness to noise \cite{needell2014paved, HadNeeMotz18}.

Corrupted data proves more of a challenge for projection-based methods: projecting onto a row with large corruption can cause the iterate to move far from the solution 
and severely disrupt convergence. Recent modifications have been designed to handle this issue, see \cite{HadNeeCorr18, DurWisdm2017, swartquantile}. In this paper we focus on the method introduced in \cite{swartquantile} and analyzed further in \cite{steinerberger2021quantilebased}, where the authors constructed a quantile-based modification of RK, QuantileRK, in which the quantile of the absolute values of a subresidual is used to detect and avoid projecting onto corrupted rows.

Here, we extend the theory and show that QuantileRK is robust to both corruptions and noise in the measurement data.  We give a theoretical result showing exponential convergence down to an error horizon, and provide experiments demonstrating the strength of the method in identifying and solving the underlying system beneath highly damaged measurement data.

\section{Proposed Method}
\subsection{Preliminaries \& Notation}

We aim to solve the consistent system $\mA \vx = \tilde{\vb}$ with access only to the observed measurement vector $\vb = \tilde{\vb} + \vb^C + \vr$, where $\vb^C$ is a sparse vector of corruptions, and $\vr$ is a vector of noise. In practice, $\vb^C$ will contain large entries, and $\vr$ small, but we make no such assumption for our theory. We define $\beta$ to be the fraction of data that is corrupted, i.e. $\beta = |\{i : b^C_i > 0\}|/m$.

We build on the foundations established in \cite{swartquantile}. To utilize results from random matrix theory, we view $\mA$ as a random matrix and make the following assumptions, that will for example hold if $\mA$ is Gaussian with normalized rows:

\noindent \textbf{Assumption 1.} All rows $\va_i$ of $\mA$ are independent, and $\sqrt{n}\va_i$ is mean zero isotropic with uniformly bounded subgaussian norm, $\norm{\sqrt{n}\va_i}_{\psi_2} \leq K$.
\newline
\noindent \textbf{Assumption 2.} Each entry $a_{ij}$ of $\mA$ has probability density function $\phi_{ij}$ satisfying $\phi_{ij}(t) \leq D\sqrt{n}$ for all $t\in \mathbb{R}$. 

We define the $q$-quantile of the absolute values of the residual, or sub-residual formed by rows in an index set $S$:
\begin{align*}
    Q_q(\vx) &= q-\text{quantile}\{|b_i - \langle \va_i, \vx \rangle|: i \in [m]\} \\
    Q_q(\vx, S) &= q-\text{quantile}\{|b_i - \langle \va_i, \vx \rangle| : i \in S\}.
\end{align*}

Throughout, $C,c,c_1,c_2,\cdots$ refer to absolute constants whose values may vary line by line.

\subsection{QuantileRK}
Projecting iterates onto corrupted hyperplanes will often cause abnormally large movements. Our method detects this by taking a quantile of the residual entries of a collection of rows at each iteration, and deeming a row acceptable for projection if its residual entry is less than said quantile. Whilst the method may still project onto corrupted rows, the movement away from the solution caused by these 'bad' projections will on average be outweighed by projections onto uncorrupted rows. We present pseudocode for the method in \cref{QuantileRK}, under the assumption that $\mA$ has been standardized to have normalized rows for simplicity.

\begin{algorithm}
	\caption{QuantileRK(q)}\label{QuantileRK}
	\begin{algorithmic}[1]
		\Procedure{QuantileRK}{$\mA,\vb$, q, t, N}
		\State{$\vx_0 = 0$}
		\For{j = 1, \ldots, N}
		\State{sample $i_1, \ldots i_t \sim \text{Uniform} (1,\ldots, m)$}
		\State{sample $k\sim\text{Uniform}(1, \ldots, m)$}
		\State{compute $q_k = Q_q(\vx_{j-1}, \{i_l: l \in [t]\})$}
		\If{ $\abs{\langle \va_k, \vx_{j-1}\rangle - b_k} \leq q_k$}
		\State{$\vx_j = \vx_{j-1} - \left(\langle \vx_{j-1}, \va_k \rangle - b_k\right) \va_k$}
		\Else
		\State{$\vx_j = \vx_{j-1}$}
		
		\EndIf
		\EndFor{}
		
		\Return{$\vx_N$}
		\EndProcedure
	\end{algorithmic}
\end{algorithm}

In \cite{swartquantile}, the authors proved that for $\mA$ sufficiently tall and $\beta$ sufficiently small, QuantileRK convergences exponentially, with
\begin{equation*}
    \mathbb{E}(\norm{\ek}^2) \leq \left(1 - \frac{C_q}{n}\right)^k \norm{\einit}^2,
\end{equation*}

Our main result, \cref{maintheorem}, builds on this and shows that the addition of noise does not harm the convergence rate, and exponential convergence is still achieved up to a horizon proportional to the size of the noise.

\begin{theorem}
\label{maintheorem}
Let the linear system be defined by the standardized random matrix $\mA \in \mathbb{R}^{m \times n}$ satisfying Assumptions 1 and 2. Assume that $\beta \leq \min(cq, 1-q)$, and that $m \geq Cn$. Then with high probability, the iterates produced by QuantileRK, with $q \in (0,1)$, where in each iteration the quantile is computed using the full residual, and initialized with arbitrary $\vx_0 \in \mathbb{R}^n$, satisfy
\begin{equation}
\mathbb{E}(\norm{\ek}^2) \leq \left(1 - \frac{C_q}{n}\right)^k \norm{\einit}^2 + \frac{2n}{c_1}\norm{\vr}_\infty^2.
\end{equation}
\end{theorem}

\begin{remark}
It is natural to ask whether one may consider some of the larger entries in $\vr$ as corruptions, by increasing $\beta$, leading to a smaller error horizon. This is possible, but there is a tradeoff: increasing $\beta$ forces a decrease in $q$, which slows convergence. The effectiveness will be application dependent: if the distribution of noise is concentrated, it would take a significant increase in $\beta$ to see a decrease in the error horizon, leading to substantially slower convergence; however, if the noise has large spikes, increasing $\beta$ may be worthwhile.
\end{remark}

\subsection{Proof of Main Result}

We follow the proof of the main QuantileRK convergence result from \cite{swartquantile} closely, making necessary alterations for the presence of noise throughout. We firstly present a modified version of Remark 3 from said paper:

\begin{lemma}
\label{lem:residualbound}
Let $\alpha \in (0,1]$, let the random matrix $\mA \in \mathbb{R}^{m \times n}$ satisfy Assumption 1, and let $\xopt$ be the solution to the consistent system $\mA \vx = \tilde{\vb}$. Then if $m \geq n$, there exists a constant $C_K >0$ so that with probability at least $1 - 2e^{-m}$, for every $\vx \in \mathbb{R}^{n}$ the bound
\[
|\langle \va_i, \vx \rangle - b_i| \leq \frac{C_K}{\alpha \sqrt{n}}\norm{\vx - \xopt} + \norm{r}_{\infty}
\]
holds for all but at most $(\alpha + \beta)m$ indices $i$.
\end{lemma}
\begin{proof}
Applying (\cite{swartquantile}, Proposition 2) with the unit vector $(\vx - \xopt)/\norm{\vx - \xopt}$, excluding the $\beta m$ corrupted rows, yields
\begin{equation*}
|\langle \va_i, \vx \rangle - \langle \va_i, \xopt \rangle| \leq \frac{C_K}{\alpha\sqrt{n}}\norm{\vx - \xopt}
\end{equation*}
for at most $(\alpha + \beta)m$ indices $i$. For each $i$ for which the above holds, we have $\langle \va_i, \xopt \rangle = \tilde{b}_i = b_i - r_i$ (i.e., $b_i^C = 0$). Then the right hand side can be written as
\begin{align*}
    |\langle \va_i, \vx \rangle - \langle \va_i, \xopt\rangle| &= |\langle \va_i, \vx \rangle - b_i + r_i| \\
    &\geq |\langle \va_i, \vx \rangle - b_i| - |r_i| \\
    &\geq |\langle \va_i, \vx \rangle - b_i| - \norm{\vr}_\infty.
\end{align*}
Combining the inequalities yields the result.
\end{proof}

Taking $\alpha \leq 1 - q - \beta$ immediately gives the following corollary, showing that the quantiles are well-concentrated:

\begin{corollary}
\label{qconcentrates}
Under the same assumptions as \cref{lem:residualbound}, and taking $\alpha \leq 1 - q - \beta$, we have
\begin{equation*}
    \mathbb{P}\left(Q_q(\vx) \leq \frac{C_{\alpha}\norm{\vx - \vx^\ast}}{\sqrt{n}} + \norm{\vr}_{\infty}\right) \geq 1 - 2e^{-m}.
\end{equation*}
\end{corollary}

We are now ready to prove \cref{maintheorem}.

\begin{proof}[Proof of \cref{maintheorem}]
Let $\mathcal{E}_{Accept}(k)$ denote the event that we sample a row that with residual less than the computed quantile at the $k$th iteration. It is clear that we have $\mathbb{P}(\mathcal{E}_{Accept}(k)) = q$.

Let $J$ be a collection of indices of size $2\beta m$, containing all corrupted indices and at least $\beta m$ acceptable indices. Then split all acceptable indices into two subsets: those inside $J$, denoted by $I_1$, and those outside of $J$, denoted by $I_2$. Let $\mathcal{E}_L^k$ denote the event that at the $k$-th iteration an index in sampled from $L \subset [m]$. We argue that the possible damage to convergence caused by projecting onto a corrupted row in $I_1$ is outweighed by the movement towards the solution caused by projecting onto a row in $I_2$.

Observe firstly that
\begin{multline}
\label{eqn:qsplit}
    \mathbb{E}_k(\norm{\ekpo}^2) = q\mathbb{E}_k(\norm{\ekpo}^2|\mathcal{E}_{Accept}(k+1)) + \\ (1-q)\norm{\ek}^2,
\end{multline}
since we have no update to our iterate if the sampled row was not acceptable.

We now deal with $\mathbb{E}_k(\norm{\ekpo}^2 | \mathcal{E}_{Accept}(k+1))$ by splitting into two cases; sampling a row from $I_1$ or from $I_2$. Note that the probability of sampling an index from $I_1$, conditioned on $\mathcal{E}_{Accept}(k+1)$, $p_J$, satisfies $p_J \leq 2\beta m/qm \leq 2\beta/q$.

Firstly, if we sample from $I_2$, the iterate $\xkpo$ is obtained by performing an iteration of standard RK on the noisy system $\mA_{I_2} \vx = \tilde{\vb}_{I_2} + \vr_{I_2}$. Noting that $I_2$ has size at least $(q-\beta)m$, Proposition 2 from \cite{swartquantile} (with $\alpha = q-\beta$) yields that $\sigma_{min}(\mA_{I_2}) \geq C_{\alpha, D}\sqrt{m/n}$ with high probability, provided that $\mA$ is tall enough. Furthermore since $\mA$ has normalized rows, we have $\norm{\mA_{I_2}}_F \geq \sqrt{(q-\beta)m}$. Thus
\begin{equation*}
    \kappa(\mA_{I_2}) \geq C_{q,D}\sqrt{n}.
\end{equation*}
Then by the analysis of RK with noise in \cite{needellnoisyrk}, we have that 
\begin{align*}
    \mathbb{E}_k(\norm{\ekpo}^2|\mathcal{E}_{I_2}^{k+1}) &\leq \left(1 - \frac{c_1}{n}\right)\norm{\ek}^2 + \norm{\vr_{I_2}}_{\infty}^2 \\
    &\leq \left(1 - \frac{c_1}{n}\right)\norm{\ek}^2 + \norm{\vr}_{\infty}^2.
\end{align*}

% Now, if $\beta = 0$ (i.e. we have no corruptions) then we obtain
% \begin{align*}
%     \begin{split}
%     \mathbb{E}_k(\norm{\ekpo}^2) &\leq q\left(1 - \frac{c_1}{n}\right)\norm{\ek}^2 + q\norm{\vr}_{\infty}^2 + \\
%     & \qquad (1-q)\norm{\ek}^2
%     \end{split}\\
%     &= \left(1 - \frac{qc_1}{n}\right) \norm{\ek}^2 + q\norm{\vr}_\infty^2.
% \end{align*}

% Inductively, we obtain
% \begin{align*}
%     \mathbb{E}(\norm{\ek}^2) &\leq \left(1 - \frac{qc_1}{n}\right)^k \norm{\ve_0}^2 + \sum_{j=0}^{k-1}\left(1 - \frac{qc_1}{n}\right)q\norm{\vr}_\infty^2 \\
%     &= \left(1 - \frac{qc_1}{n}\right)^k \norm{\ve_0}^2 + \frac{n}{c_1}\norm{\vr}_\infty^2.
% \end{align*}

The $\beta = 0$ case (i.e., when we have no corruptions) follows immediately from this and \cref{eqn:qsplit}. In the case where $\beta > 0$, i.e., when $I_1$ is not empty, we consider the possibility that we sample from $I_1$. Our update will take the form $\xkpo = \xk - h_i \va_i$, where $|h_i| \leq Q_q(\xk)$, and so we have
\begin{multline*}
    \mathbb{E}_k(\norm{\ekpo}^2|\mathcal{E}_{I_1}^{k+1}) \leq
    \norm{\ek}^2 + Q_q(\xk)^2 + 
    \\2Q_q(\xk)\mathbb{E}_k(|\langle \ek, \va_i \rangle | i \sim \text{Unif}(I_1)).
\end{multline*}
To continue estimating, note that we have by \\ 
(\cite{swartquantile}, Lemma 4), with probability $1 - 2e^{-cm}$,  
\begin{align*}
    \mathbb{E}_k(|\langle \ve_k, \va_i \rangle | | i \sim \text{Unif}(I_1)) &= \frac{1}{|I_1|}\sum_{i \in I_1}|\langle \ek, \va_i \rangle| \\
    &\leq \frac{C\norm{\ek}}{\sqrt{\beta n}}.   
\end{align*}
Then using this and the result of \cref{qconcentrates}:
\begin{multline*}
    \mathbb{E}_k(\norm{\ekpo}^2 | \mathcal{E}_{I_1}^{k+1})\leq \left(1 + \frac{\sqrt{\beta}c_2 + c_3}{n\sqrt{\beta}}\right)\norm{\ek}^2 \\
    + \left(\frac{c_4 \sqrt{\beta} + c_5}{\sqrt{n\beta}}\right)\norm{\vr}_\infty \norm{\ek} + \norm{\vr}_\infty^2.
\end{multline*}
We can now estimate $\mathbb{E}_k(\norm{\ekpo}^2 | \mathcal{E}_{Accept}(k+1))$ as follows:
\begin{multline*}
    \mathbb{E}_k(\norm{\ekpo}^2 | \mathcal{E}_{Accept}(k+1)) = p_J\mathbb{E}_k(\norm{\ekpo}^2 | \mathcal{E}_{I_1}^{k+1}) \\
    + (1-p_J)\mathbb{E}_k(\norm{\ekpo}^2 | \mathcal{E}_{I_2}^{k+1}) \\
    \leq \left(1 - \frac{c_1}{n} + p_J\left(\frac{\sqrt{\beta}(c_1 + c_2) + c_3}{n\sqrt{\beta}}\right)\right)\norm{\ek}^2 \\
    + p_J \left(\frac{c_4 \sqrt{\beta} + c_5}{\sqrt{n\beta}}\right)\norm{\vr}_\infty \norm{\ek} + \norm{\vr}_\infty^2.
\end{multline*}
To handle the $\norm{\vr}_\infty\norm{\ek}$ term we split into two cases. The motivation is that when our error is large relative to the noise, the quantile can detect corruptions well, whereas when the error is small relative to the noise, our movement will be small.  Firstly, if $\sqrt{n}\norm{\vr}_\infty \leq \norm{\ek}$ (i.e. when our error is large), we have
\begin{multline*}
    \mathbb{E}_k(\norm{\ekpo}^2 | \mathcal{E}_{Accept}(k+1)) \leq \left(1 - \frac{c_1}{n} + \right. \\
    \left. p_J\left(\frac{\sqrt{\beta}(c_1 + c_2 + c_4) + c_3 + c_5}{n\sqrt{\beta}}\right)\right) \norm{\ek}^2 + \norm{\vr}_\infty^2 \\
    \leq \left(1 - \frac{0.5c_1}{n}\right)\norm{\ek}^2 + \norm{\vr}_\infty^2
\end{multline*}
for small enough $\beta$ (we need $\sqrt{\beta} \leq cq$). On the other hand, if $\sqrt{n}\norm{\vr}_\infty \geq \norm{\ek}$, we have
\begin{multline*}
    \mathbb{E}_k(\norm{\ekpo}^2 | \mathcal{E}_{Accept}(k+1)) \leq \left( 1- \frac{c_1}{n} + \right. \\ 
    \left. p_J\frac{\sqrt{\beta}(c_1 + c_2) + c_3}{n\sqrt{\beta}}\right)\norm{\ek}^2 + p_J \left(\frac{c_4 \sqrt{\beta} + c_5}{\sqrt{\beta}}\right)\norm{\vr}_\infty^2 \\
    \leq \left(1 - \frac{0.5c_1}{n}\right)\norm{\ek}^2 + \norm{\vr}_\infty^2,
\end{multline*}
again for $\sqrt{\beta} \leq cq$ sufficiently small.

We may now substitute our expressions into \cref{eqn:qsplit} to obtain our per-iteration guarantee:

\begin{equation*}
    \mathbb{E}_k(\norm{\ekpo}^2) \leq  \left(1 - \frac{0.5 qc_1}{n}\right)\norm{\ek}^2 + q\norm{\vr}_\infty^2.
\end{equation*}

By induction, we obtain our overall guarantee:

\begin{multline*}
    \mathbb{E}(\norm{\ek}^2) \leq \left(1 - \frac{0.5 qc_1}{n}\right)^k\norm{\ve_0}^2 + \\
    \sum_{j=0}^{k-1}\left(1 - \frac{0.5 qc_1}{n}\right)^j q\norm{\vr}_\infty^2 \\
    \leq \left(1 - \frac{0.5qc_1}{n}\right)\norm{\ve_0}^2 + \frac{2n}{c_1}\norm{\vr}_\infty^2.
\end{multline*}
\end{proof}
\section{Experimental Results}
Experiments are performed on $2000 \times 100$ standardized Gaussian matrices $\mA$. We sample a Gaussian $\vx^\ast \in \mathbb{R}^{100 \times 1}$, compute $\vb = \mA\vx$, and then corrupt a fraction $\beta$ of the rows of $\vb$ by adding corruptions of size to be specified. We add noise $\vr \in \mathbb{R}^{2000 \times 100}$ with Uniform$(-0.02, 0.02)$ entries, and apply QuantileRK to the resulting system. At each iteration $400$ rows are sampled, from which the subresidual is computed.
 
In \cref{fig:range_of_corr_sizes} we take $q = 0.7$, $\beta = 0.2$, and corrupt the already noisy system with corruptions taken from Uniform$(-k,k)$ for a range of $k$. We see that when corruptions are large relative to the noise, they are better detected by the quantile, faster convergence is achieved. When corruptions are small, they do not disrupt convergence enough to break the method, and convergence is achieved down to the error horizon. 
%\dn{Minor note, I prefer to say ``it" converges (I hope we are not converging) and ``the" error horizon. In general, I prefer passive voice when possible (maybe sweep for this quickly).}

\begin{figure}[H]
    \centering
    \includegraphics[width=.4\textwidth]{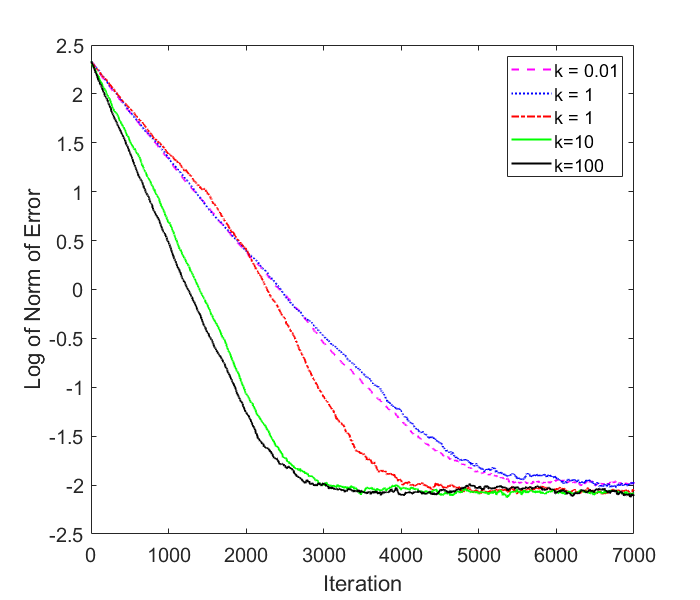}
    \caption{Convergence of QuantileRK($0.7$) with Uniform$(-k,k)$ corruptions, for a range of $k$, and Uniform$(-0.02, 0.02)$ noise.}
    \label{fig:range_of_corr_sizes}
\end{figure}

\begin{figure}
    \centering
    \includegraphics[width=.4\textwidth]{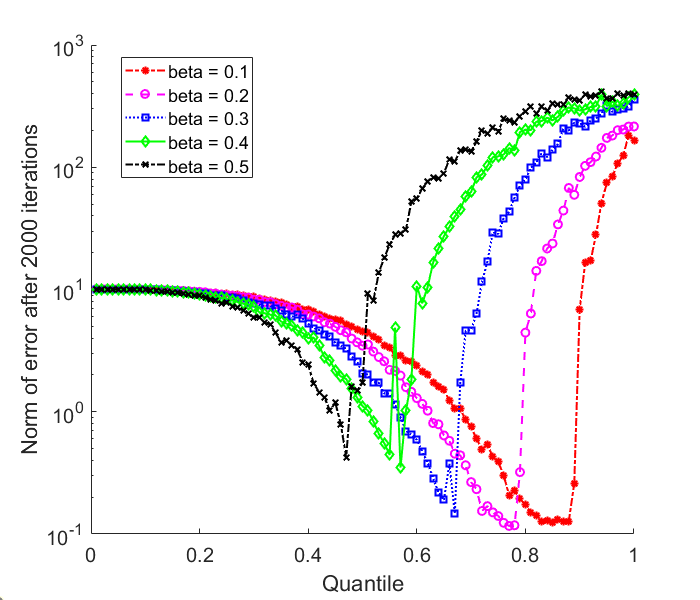}
    \caption{$\norm{\vx_{2000} - \xopt}$ for a range of corruption rates $\beta$ and quantile choices $q$.}
    \label{fig:range_of_q}
\end{figure}
%\dn{They might get picky and ask for plots that can be viewed in black and white (i.e. have different line types). If you saved the figs this should be easy?}
We would like to take $q$ as large as possible so that we may sample rows yielding large movement, but we must take $q < 1-\beta$ to avoid corrupted rows. In \cref{fig:range_of_q} we plot the normed error after 2000 iterations for a range of $q$ and $\beta$, and we see that we can be very aggressive with our choice of $q$: we are able to take it very close to $1-\beta$, and should do so to accelerate convergence.

In \cref{fig:err_vs_thresh} we simulate 100 trials, and compare the error (after 5000 and 10000 iterations respectively) to the predicted horizon. Indeed, our results show that the predicted horizon is closely respected.

\begin{figure}[H]
    \centering
    \includegraphics[width=.4\textwidth]{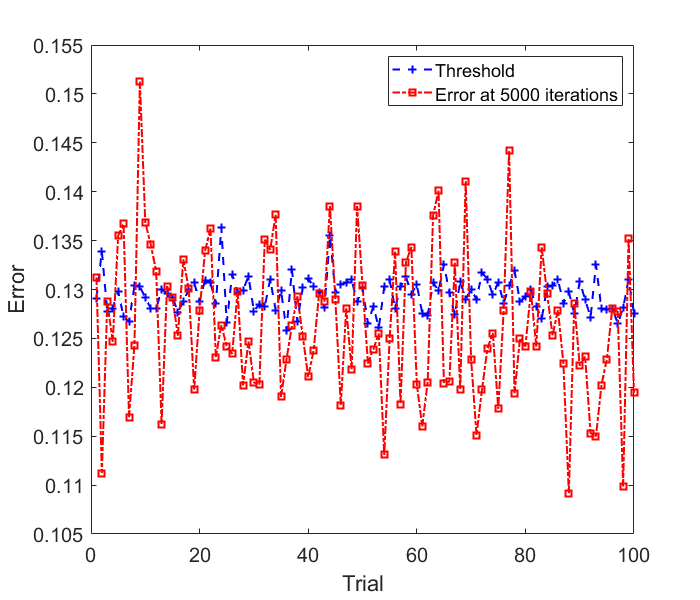}
    \caption{Comparing $\norm{\vx_{5000} - \xopt}$ with predicted error horizon.}
    \label{fig:err_vs_thresh}
\end{figure}

\section{Conclusion and Future Work}\label{conclusion}

We have shown, both theoretically and empirically, that QuantileRK is a powerful method for solving linear systems where measurement data has been damaged by both corruptions and noise. We believe that this method will prove tractable in practice, as corruption and noise are ubiquitous in real-world data.

We are interested in pursuing quantile-based modifications to other projection-based iterative methods, see \cite{gowersap} for a general framework, and also in relaxing the conditions placed on our system: see \cite{steinerberger2021quantilebased} for some work in this direction.

\section*{Acknowledgment}
The authors are grateful for the support of NSF BIGDATA \#1730325 and NSF DMS \#2011140.

\bibliographystyle{IEEEtran}
\bibliography{main}

\end{document}